\documentclass{article}
\usepackage[utf8]{inputenc}
\usepackage[english]{babel}
\usepackage{amsthm, amsmath, amsfonts}
\usepackage{color}
\usepackage{fullpage} 
\usepackage{enumitem}
\usepackage{comment}

\newtheorem{theorem}{Theorem}
\newtheorem{prop}[theorem]{Proposition}
\newtheorem{lemma}[theorem]{Lemma}

\newtheorem*{conj*}{Conjecture}
\newtheorem{cor}[theorem]{Corollary}
\newtheorem{obs}[theorem]{Observation}
\newtheorem{remark}{Remark}

\title{Rainbow numbers for $x_1+x_2=kx_3$ in $\mathbb{Z}_n$}
\author{Erin Bevilacqua\thanks{Department of Mathematics, Pennsylvania State University, State College, PA 16801, USA, eib5092@psu.edu. NSA Grant H98230-18-1-0043.}, Samuel King\thanks{Department of Mathematics, University of Rochester, Rochester, NY 14627, sking19@u.rochester.edu. NSA Grant H98230-18-1-0043.}, J\"{u}rgen Kritschgau\thanks{Department of Mathematics, Iowa State University, Ames, IA 50011, USA, jkritsch@iastate.edu.},\\ Michael Tait\thanks{Departemnt of Mathematics, Carnegie Mellon University, Pittsburgh, PA 15289, mtait@cmu.edu. Research supported by NSF Grant DMS-1606350.}, Suzannah Tebon\thanks{Department of Mathematics, Beloit College, Beloit, WI 53589, tebonsr@beloit.edu. NSA Grant H98230-18-1-0043}, Michael Young\thanks{Department of Mathematics, Iowa State University, Ames, IA 50011, USA, myoung@iastate.edu. Research supported by NSF Grant DMS-1719841 and NSA Grant H98230-18-1-0043.}}
\begin{document}
\maketitle

\begin{abstract}
In this work, we investigate the fewest number of colors needed to guarantee a rainbow solution to the equation $x_1 + x_2 = k x_3$ in $\mathbb{Z}_n$.
This value is called the Rainbow number and is denoted by $rb(\mathbb{Z}_n, k)$ for positive integer values of $n$ and $k$.
We find that $rb(\mathbb{Z}_p, 1) = 4$ for all primes greater than $3$ and that $rb(\mathbb{Z}_n, 1)$ can be deterimined from the prime factorization of $n$.
Furthermore, when $k$ is prime, $rb(\mathbb{Z}_n, k)$ can be determined from the prime factorization of $n$.
\end{abstract}

\section*{Introduction}
Let $\mathbb{Z}_n$ be the cyclic group of order $n$, and let an \textit{r-coloring} of $\mathbb{Z}_n$ be a function $c : \mathbb{Z}_n \to [r]$ where $[r] := \{1,...,r\}$.
In this paper, we assume that each $r$-coloring is \textit{exact} (surjective).
Given an exact $r$-coloring, we define $r$ \textit{color classes} $C_i = \{x\in\mathbb Z_n\mid c(x)=i\}$ for $1 \leq i \leq r$.
Occasionally, when convenient, we will use $R$, $G$, $B$, and $Y$ to denote the colors or the color classes red, green, blue, and yellow, respectively.

Fix an integer $k$. Let a \emph{triple} $(x_1, x_2, x_3)$ be any three elements in $\mathbb{Z}_n$ which are a solution to $x_1 + x_2 \equiv kx_3 \mod n$.
When $k=1$, we will call these triples \textit{Schur triples}.
Such a triple is called a \textit{rainbow triple} under a coloring $c$ when  $c(x_1) \not = c(x_2)$, $c(x_1) \not = c(x_3)$, and $c(x_2) \not = c(x_3)$.
Consequently, a coloring will be called \textit{rainbow-free} when there does not exist a rainbow triple in $\mathbb{Z}_n$ under $c$.

The \emph{rainbow number} of $\mathbb{Z}_n$ given $x_1+x_2=kx_3$, denoted $rb(\mathbb Z_n,k)$, is the smallest positive integer $r$ such that any $r$-coloring of $\mathbb{Z}_n$  admits a rainbow triple. By convention, if such an integer does not exist, we set $rb(\mathbb Z_n,k)=n+1$.
A \textit{maximum} coloring is a rainbow-free $r$-coloring of $\mathbb{Z}_n$ where $r = rb(\mathbb{Z}_n,k)-1$.

For a coloring $c$ of $\mathbb{Z}_{st}$, the $i^{th}$ \textit{residue class} modulo $t$ is the set of all the elements in $\mathbb{Z}_{st}$ which are congruent to $i \mod t$. 
Denote each residue class as $R_i = \{j \in \mathbb{Z}_{st} | j \equiv i \mod t \}$.
We say the $i^{th}$ \textit{residue palette} modulo $t$ is the set of colors which appear in the $i^{th}$ \textit{residue class}, and we will denote each palette as $P_i = \{c(j)|j \equiv i \mod t \}$.

Rainbow numbers for the equation $x_1+x_2=2x_3$, for which the solutions are $3$-term arithmetic progressions, have been studied in \cite{BSY}, \cite{BEHHKKLMSWY}, \cite{JLMNR}, and \cite{Y}. These problems are historically rooted in Roth's Theorem, Szemer\'edi's Theorem, and van der Waerden's Theorem.
The first half of our paper explores the rainbow numbers of $\mathbb{Z}_n$ given the Schur equation, $x_1+x_2=x_3$.
We rely on the work of Llano and Montenjano in \cite{LM}, Jungi\'c et al. in \cite{JLMNR}, and Butler et al. in \cite{BEHHKKLMSWY} to prove exact values for $rb(\mathbb{Z}_n,1)$ in terms of the prime factorization of $n$. Our results are an extension to the results in \cite{BSY}, \cite{JLMNR}, and \cite{Y}.

\begin{theorem}\label{SchurPrimes4}
For a prime $p \geq 5$, $rb(\mathbb{Z}_p, 1) = 4.$
\end{theorem}

\begin{remark}
It can be deduced through inspection that $rb(\mathbb{Z}_2,1) = rb(\mathbb{Z}_3,1) = 3$.
\end{remark}

Theorem \ref{SchurPrimes4} gives exact values for $rb(\mathbb Z_p,1)$ where $p$ is prime. Therefore, Theorems \ref{SchurFactorization} and \ref{SchurPrimes4} give exact values for $rb(\mathbb{Z}_n,1)$. The proof for Theorem \ref{SchurFactorization} is at the end of Section \ref{SectionSchurUpperBound}.

\begin{theorem}\label{SchurFactorization}
For a positive integer $n$ with prime factorization $n = p_1^{\alpha_1} \cdot p_2^{\alpha_2} \cdots p_m^{\alpha_m}$, $$\text{rb}(\mathbb Z_n,1) = 2+\sum_{i=1}^{m} \Big( \alpha_i(\text{rb}(\mathbb{Z}_{p_i},1)-2) \Big).$$
\end{theorem}

We continue by considering the equation $x_1+x_2=px_3$ for any prime $p$. Many of the techniques for the $k=1$ case generalize. However, there are complications. If we let the prime factorization of $n$ be $n=p^\alpha \cdot q_1^{\alpha_1} \cdots q_m^{\alpha_m}$, then we can produce a recursive formula for $rb(\mathbb{Z}_n,p)$ detailed in Theorem \ref{GenKFormula}.

\begin{theorem}\label{thmpneqq}
Let $p,q$ be distinct and prime. Then $rb(\mathbb Z_q,p)=4$ if and only if $p,q$ do not satisfy either of the following conditions:

\begin{enumerate}
\item $p$ generates $\mathbb Z_q^*$,
\item $|p|=(q-1)/2$ in $\mathbb Z_q^*$ and $(q-1)/2$ is odd.
\end{enumerate}

Otherwise, $rb(\mathbb Z_q,p)=3$.
\end{theorem}

\begin{theorem}\label{prime}
For $p\geq 3$ prime and $\alpha\geq 1$, 
$$ rb(\mathbb{Z}_{p^\alpha},p)=\begin{cases}3&p=3,\alpha=1\\
4& p=3, \alpha\geq 2\\
\frac{p+1}{2}+1 &p\geq 5\end{cases}$$
\end{theorem}

The values for $rb(\mathbb Z_{2^\alpha},2)$ are resolved in \cite{BSY}. In conjunction with Theorems \ref{thmpneqq} and \ref{prime}, Theorem \ref{GenKFormula} determines exact values for $rb(\mathbb Z_n,p)$. The proof for Theorem \ref{GenKFormula} is at the end of Section \ref{SectionGenKUpperBound}.

\begin{theorem}\label{GenKFormula}
Let $n$ be a positive integer, and let $p$ be prime. Let $n$ have prime factorization $n=p^\alpha \cdot q_1^{\alpha_1} \cdots q_m^{\alpha_m}$.
Then $$rb(\mathbb{Z}_n, p) = rb(\mathbb{Z}_{p^{\alpha}}, p) + \sum_{i=1}^m \Big( \alpha_i(rb(\mathbb{Z}_{q_i}, p)-2) \Big).$$ In the case that $\alpha=0$, let $rb(\mathbb Z_{p^\alpha},p)=2$.
\end{theorem}

\section{Schur Triples}\label{SectionSchur}
Section 1 is dedicated to proving Theorem \ref{SchurFactorization}.
In Section \ref{SectionSchurPrimes} we introduce the idea of a dominant color to describe the structural properties of colorings of $\mathbb{Z}_p$.
Additionally, we prove Proposition \ref{MinPrime}, the Schur triple counterpart of Theorem 3.2 in \cite{JLMNR}. We use Proposition \ref{MinPrime} to prove Theorem \ref{SchurPrimes4}, concluding Section 1.1. In Section \ref{SectionSchurLowerBound} we show that the lower bound of $rb(\mathbb{Z}_n,1)$ can be determined by the prime factorization of $n$.
The equivalent upper bound is proved in \ref{SectionSchurUpperBound}. Combining Sections \ref{SectionSchurLowerBound} and \ref{SectionSchurUpperBound} proves Theorem \ref{SchurFactorization}. 

\subsection{Schur Triples in $\mathbb{Z}_p$, $p$ prime}\label{SectionSchurPrimes}
Let $c$ be a coloring of $\mathbb{Z}_n$. We say a sequence $S_1,S_2,\dots, S_k$ of colors \emph{appears at position $i$} if $c(i)=S_1, c(i+1)=S_2, \dots, c(i+k-1)=S_k$. A sequence is \emph{bichromatic} if it contains exactly two colors. A color $R$ is \emph{dominant} if for $S = \{c(x):i \leq x \leq j, i < j\}$, $|S| = 2$ implies $R \in S$. That is, $R$ appears in every bichromatic string. Using dominant colors to derive a contradiction is used in \cite{JLMNR}. We also use this idea to describe the structure of rainbow-free colorings of $\mathbb{Z}_p$. However, we must show that a dominant color exists.

\begin{lemma}\label{DomColor} There exists a dominant color in every rainbow-free coloring of $\mathbb{Z}_n$. 
Furthermore, $c(1)$ is dominant.
\end{lemma}
\begin{proof}
Let $c$ be a rainbow-free coloring of $\mathbb{Z}_n$.
Note that $(1, i, i+1)$ is a Schur triple for all $i \not \in \{0,1\}$.
Since $c$ is rainbow-free, either $c(i)=c(i+1)$, $c(1)=c(i)$, or $c(1)=c(i+1)$.
Thus, if $c(i)\neq c(i+1)$, then $c(1)$ must appear on either $i$ or $i+1$.
This implies that $c(1)$ is dominant.
\end{proof}

An immediate result from this lemma is that any color which doesn't appear on $1$ must be adjacent to itself or the dominant color.
Now we can relate the structure of our coloring to the presence of a rainbow triple. Without loss of generality, let $c(1)=R$ be dominant.

\begin{lemma}\label{NoBBGG}
Let $c$ be an $r$-coloring of $\mathbb{Z}_n$ with $r\geq 3$. If $BB$ and $GG$ appears in $c$, then there exists a rainbow Schur triple in $c$.
\end{lemma}
\begin{proof}
Let $c$ be an $r$-coloring of $\mathbb{Z}_n$ with $r\geq 3$ such that $BB$ and $GG$ appears in $c$.
Without loss of generality, assume $R$ is dominant, and $c$ contains $BB$ and $GG.$
Then, the sequence $BBR$ must appear at some position $i$ and the sequence $GGR$ must appear at some position $j$.

Consider the Schur triple $(i, j+2, i+j+2)$. Since $c(i)=B$, and $c(j+2)=R$, then either $c$ contains a rainbow Schur triple, or $c(i+j+2)$ is $R$ or $B$.
Assume the second case, and consider the Schur triple $(i+2, j, i+j+2)$.
Since $c(i+2)=R$, and $c(j)=G$ then either $c$ contains a rainbow Schur triple or $c(i+j+2)$ is $R$.
Again, assume the second case, and finally consider the triple $(i+1, j+1, i+j+2)$.
Since $c(i+1)=B$, $c(j+1)=G$, and $c(i+j+2)=R$, this triple is rainbow.
Therefore, $c$ contains a rainbow Schur triple.
\end{proof}

Therefore, if $c$ is a rainbow-free coloring of $\mathbb{Z}_n$ with $R$ dominant,  either $GG$ or $BB$ can appear in $c$, but not both.
Next we show that there are ways to re-order colorings while maintaining whether or not Schur triples are rainbow.

\begin{lemma}\label{C-Hat}
Let $c$ be an $r$-coloring of $\mathbb{Z}_n$.
If $m$ is relatively prime to $n$, then $c$ has a rainbow Schur triple if and only if $\hat{c}(x):=c(mx)$ contains a rainbow Schur triple. Additionally, the cardinality of each color class will be maintained.
\end{lemma}
\begin{proof}
Let $(x_1, x_2, x_3)$ be a triple in $c$.
By definition, $x_1 + x_2 = x_3$ in $\mathbb{Z}_n$ is equivalent to
\begin{align*}
x_1 + x_2 = sn + r \\
x_3 = tn + r,
\end{align*}
as equations in the integers for some $s,t\in \mathbb{Z}$. Multiply both equations by $m$ to get
\begin{align*}
mx_1 + mx_2 = msn + mr \\
mx_3 = mtn + mr
\end{align*}
Therefore, $mx_1 + mx_2 \equiv mr \mod n$, and $mx_3 \equiv mr \mod n$, so $mx_1 + mx_2 \equiv mx_3 \mod n$. 
Thus, $(mx_1, mx_2, mx_3)$ is rainbow in $\hat{c}$ if and only if $(x_1, x_2, x_3)$ is rainbow in $c$.

Finally, the last statement of Lemma \ref{C-Hat} follows from the fact that if $m$ is relatively prime to $n$, then the map $F: x \mapsto mx$ is a bijection.
\end{proof}

Our next result is the Schur equation counterpart to Theorem 3.2 in \cite{JLMNR}.

\begin{prop}\label{MinPrime}
Let $p$ be prime.
Then every $3$-coloring $c$ of $\mathbb{Z}_p$ with $\min(|R|,|G|,|B|) > 1$ contains a rainbow Schur triple.
\end{prop}
\begin{proof}
For the sake of contradiction, assume that $c$ is a rainbow-free $3$-coloring of $\mathbb{Z}_p$ and min$(|R|,|G|,|B|) > 1$.
Without loss of generality, assume that $|R|=\textrm{min}(|R|,|G|,|B|)$.
Since there are at least two elements of $\mathbb{Z}_p$ colored $R$, there exists a minimal element $1\leq i \leq p-1$ such that $c(i)=R$
Because $p$ is prime, $i$ is relatively prime to $p$ and $i$ has a multiplicative inverse. 
Let $\hat{c}(x):=c(ix)$ so that $\hat{c}(1)= R$.
Therefore, by Lemma \ref{DomColor}, $R$ is dominant in $\hat{c}$.
By Lemma \ref{NoBBGG}, $BB$ and $GG$ cannot both appear in $\hat{c}$.
Without loss of generality, assume that $GG$ does not appear in $\hat{c}$.
Because $R$ is dominant, $R$ must follow each $G$, so $|R| \geq |G|$.
Furthermore, $BR$ must appear in $\hat{c}$.
This implies that $|R| \geq |G|+1$ in $\hat c$ which implies $|R| \geq |G|+1$ in $c$ by Lemma \ref{C-Hat}.
This contradicts our assumption that $|R| =\min(|R|,|G|,|B|)$.
\end{proof}

\begin{lemma}\label{symmetry}
If $c$ is a rainbow-free $r$-coloring of $\mathbb{Z}_p$ for a prime $p$ with $r>2$, then $c(x)=c(-x)$.
\end{lemma}
\begin{proof}
Let $c$ be a rainbow-free $r$-coloring of $\mathbb{Z}_p$. For the sake of contradiction, assume that there exists $i,-i$ with $c(i) \neq c(-i)$.
Without loss of generality, let $c(i)=R$ and $c(-i)=G$.
Now, let $\hat{c}(x):=c(ix)$ and let $\bar{c}(x):=c(-ix)$. By Lemma \ref{C-Hat}, $\hat{c}$ and $\bar{c}$ are both rainbow-free.
Since $\hat{c}(1)=c(i)=R$ and $\bar{c}(1)=c(-i)=G$, $R$ is dominant in $\hat{c}$, and $G$ is dominant in $\bar{c}$.
Notice that $\hat{c}(x)=\bar{c}(-x)$, so if two colors are adjacent at some position in $\hat{c}$, then they are also adjacent at some position in $\bar{c}$.
Thus, since $G$ is dominant in $\bar{c}$, $G$ must also appear in every bichromatic sequence in $\hat{c}$, and, consequently, $G$ is also dominant in $\hat{c}$.
If both $R$ and $G$ are dominant in $\hat{c}$, then $\hat{c}$ must only contain $R$ and $G$, and $r=2$; this is a contradiction.
\end{proof}

Note that this lemma shows that the coloring from $1$ to $p-1$ must be symmetric in a rainbow-free coloring of $\mathbb{Z}_p$.

\begin{remark}
For any prime $p \geq 5$, $\mathbb{Z}_p$ can be colored with three colors by coloring zero uniquely and coloring $1$ to $p-1$ with two colors in any way such that $c(x) = c(-x)$ for all $x$. This coloring is rainbow-free since any three group elements which witness three colors must contain $0$, and in order to make a Schur triple of three distinct elements where one of the elements is $0$ the other two elements must be $x$ and $-x$ for some $x$ (see also Corollary 2 in \cite{LM}).
\end{remark}

 Now we have enough information about the structure of rainbow-free colorings to prove Theorem \ref{SchurPrimes4}. A color class $C$ is \textit{singleton} if $|C|=1$.

\begin{proof} [Proof of Theorem \ref{SchurPrimes4}]
For the sake of contradiction, suppose that $r+1=rb(\mathbb Z_p, 1)>4$ for a prime $p \geq 5$, and let $c$ be a rainbow-free $r$-coloring of $\mathbb{Z}_p$ with $r>3$. 
Note that since $c$ is rainbow-free, at least one of the color classes in $c$ must contain more than one element.
Partition the color classes of $c$ into three sets to define $\hat{c}$, an exact $3$-coloring of $\mathbb{Z}_p$.
We use the union of the color classes within each part of the partition as the color classes for $\hat{c}$.
Since we are concatenating colors, $\hat{c}$ is also rainbow-free. 
By Proposition \ref{MinPrime}, regardless of how the color classes of $c$ are partitioned, there exists some color class in $\hat{c}$ with exactly one element.
If $r \geq 5$, then there exists a partition of the five or more color classes such that each color class has more than one element. Therefore, $r = 4$.

Furthermore, if two or more color classes are not singleton, then there would exist a partition of the color classes that yields no singleton color classes in $\hat{c}$. Therefore, all but one of the four color classes in $c$ must be singleton.

If there are three singleton color classes in $c$, then there exists an $x\not=0$ such that $c(x)\not=c(-x)$. This contradicts Lemma \ref{symmetry}, and $c$ cannot be rainbow-free.

Thus, there does not exist an exact rainbow-free $r$-coloring of $\mathbb{Z}_p$ for $r >3$ and $p \geq 5$.
\end{proof}

\subsection{Lower Bound}\label{SectionSchurLowerBound}
In order to prove the lower bound for $rb(\mathbb Z_n,1)$, we examine the relationship between Schur triples in $\mathbb Z_n$ and $\mathbb Z_{\frac{n}{m}}$ where $m$ divides $n$.

\begin{lemma} \label{SchurProjectTriples}
If there exists a Schur triple of form $(x_1,x_2,x_3)$ in $\mathbb{Z}_n$ where $m|x_1,x_2,x_3$ for some $m|n$, $m,n\in \mathbb{Z}$,  then there exists a Schur triple of the form $(x_1/m ,x_2/m ,x_3/m)$ in $\mathbb{Z}_\frac{n}{m}$. 
\end{lemma}
\begin{proof}
By definition, $x_1 + x_2 = x_3$ in $\mathbb{Z}_n$ implies that in the integers
\begin{align*}
x_1 + x_2 = qn + r \\
      x_3 = tn + r,
\end{align*}
for some $q,t \in \mathbb{Z}$. Divide both equations by $m$ to get

\begin{align*}
\frac{x_1}{m} + \frac{x_2}{m} = q\frac{n}{m} + \frac{r}{m} \\
      \frac{x_3}{m} = t\frac{n}{m} + \frac{r}{m}.
\end{align*}

Now we must check that $\frac{r}{m}$ is an integer. Since $m|(x_1 + x_2 - qn)$, we know $m|r$.

By definition, this means that there exists a Schur triple of the form $(x_1/m,x_2/m,x_3/m)$ in $\mathbb{Z}_\frac{n}{m}$.
\end{proof}

This shows that Schur triples can be ``projected'' from the cyclic group $\mathbb Z_n$ to a subgroup $\mathbb Z_{\frac{n}{m}}$.
Next, we will show another property of Schur triples related to the divisibility of a triple's elements by a prime.

\begin{lemma}\label{SchurPrimeDivisibility}
For a positive integer $n$ and a prime $p$, if $x_1+x_2 \equiv x_3 \mod np$, then $p$ cannot divide exactly two of $(x_1,x_2,x_3)$.
\end{lemma}
\begin{proof}
If $x_1+x_2 \equiv x_3 \mod np$, then there exist integers $c_1$, $c_2$, and $r_0$ such that $x_1+x_2=c_1np+r_0$ and $x_3=c_2np+r_0$.

Assume  that $p$ divides $x_1$ and $x_2$. Then there exist integers $c_3$ and $c_4$ such that $x_1 = c_3 p$ and $x_2 = c_4 p$.
We know there exist integers $c_5$ and $r_1$ with $0 \leq r_1 < p$ such that $x_3 = c_5 p + r_1$, so we want to show $r_1 = 0$.
Immediately, we see that $c_3 p + c_4 p = c_1 n p + r_0$ and $c_5 p + r_1 = c_2 n p + r_0$, which, after substituting for $r_0$, shows us $c_3 p + c_4 p = c_1 n p + c_5 p + r_1 - c_2 n p$.
Solving for $r_1$ gives us
\begin{equation*}
\begin{split}
r_1 &= c_3 p + c_4 p - c_1 n p - c_5 p + c_2 n p \\
&=p (c_3 + c_4 - c_1 n - c_5 + c_2 n)
\end{split}
\end{equation*}
This means that $p$ divides $r_1$, forcing $r_1 = 0$.
Thus, $p$ divides $x_3$.

Now assume $p$ divides $x_1$ and $x_3$, i.e. there exist integers $c_6$ and $c_7$ such that $x_1 = c_6 p$ and $x_3 = c_7 p$.
We know there exist integers $c_8$ and $r_2$ with $0 \leq r_2 < p$ such that $x_2 = c_8 p + r_2$, so we want to show $r_2 = 0$.
Immediately, we see that $c_6 p + c_8 p + r_2 = c_1 n p + r_0$ and $c_7 p = c_2 n p + r_0$, which, after substituting for $r_0$, shows us $c_6 p + c_8 p + r_2 = c_1 n p + c_7 p - c_2 n p$.
Solving for $r_2$ gives us
\begin{equation*}
\begin{split}
r_2 &= c_1 n p + c_7 p - c_2 n p - c_6 p - c_8 p \\
&=p (c_1 n + c_7 - c_2 n - c_6 - c_8)
\end{split}
\end{equation*}
This means that $p$ divides $r_2$, forcing $r_2 = 0$.
Thus, $p$ divides $x_2$.
By symmetry, this case is identical to the case where $p$ divides $x_2$ and $x_3$.

Therefore, we can see that if $p$ divides two elements in $(x_1,x_2,x_3)$, then $p$ must also divide the third.
\end{proof}

\begin{lemma}\label{IStep}
Let $p,t$ be positive integers with $p$ prime. If there exists a rainbow-free $r$-coloring of $\mathbb{Z}_t$, then there exists a rainbow-free $r+rb(\mathbb Z_p,1)-2$-coloring of $\mathbb Z_{pt}$.
\end{lemma}

\begin{proof}
Let $t,p$ be positive integers such that $p$ is a prime.
Assume $\hat{c}$ is a rainbow-free $r$-coloring of $\mathbb{Z}_t$.
Then let $c$ be an exact $(r+rb(\mathbb Z_p,1)-2)$-coloring (if $p=2$ or $p=3$, then $c$ is an exact $(r+1)$-coloring. Otherwise, $c$ is an exact $r+2$ coloring) of $\mathbb{Z}_{pt}$ as follows:
$$c(x):=\begin{cases}
      \hat{c}(x/p) & x \equiv 0 \mod p\\ 
      r+1 & x \equiv 1 \text{ or } p-1 \mod p \\
      r+2 & \text{otherwise}
   \end{cases}
$$
Notice that if $(x_1,x_2,x_3)$ is a Schur triple in $\mathbb{Z}_{pt}$, then there are three cases by Lemma \ref{SchurPrimeDivisibility}: $p$ divides exactly one of $(x_1,x_2,x_3)$, $p$ divides each of $(x_1,x_2,x_3)$, or $p$ divides none of $(x_1,x_2,x_3)$.

\textbf{Case 1:} The two terms $x_i,x_j$ where $i,j\in\{1,2,3\}$ that are not divisible by $p$ are either additive inverses modulo $p$ or are equal modulo $p$. Thus, $c(x_i)=c(x_j)$ and $(x_1,x_2,x_3)$ does not form a triple.

\textbf{Case 2:} The coloring of each $x_i$ is inherited from $\hat{c}$. Since $\hat{c}$ does not admit rainbow triples, we know that this triple will not be rainbow  by Lemma \ref{SchurProjectTriples}.

\textbf{Case 3:} The three integers in the triple will be colored from $\{r+1, r+2\}$, so the triple will not be rainbow. In each case, $c$ is a rainbow-free $r+rb(\mathbb Z_p,1)-2$-coloring of $\mathbb Z_{pt}$.
\end{proof}

\begin{prop} \label{SchurLowerbd}
For any positive integer $n=p_1^{\alpha_1}\cdots p_m^{\alpha_m}$,
$$rb(\mathbb{Z}_n,1) \geq 2+\sum_{i=1}^{m} \Big( \alpha_i(rb(\mathbb{Z}_{p_i},1)-2) \Big).$$
\end{prop}
\begin{proof}
If $n$ is prime, there is nothing to show. Suppose that the claim holds true for $n$ where $n$ has $N$ prime factors. 

Assume that $n=p_1^{\alpha_1}\cdots p_m^{\alpha_m}$ where $\alpha_1+\cdots+\alpha_m=N+1$. By the induction hypothesis, there exists a rainbow-free $r$-coloring of $\mathbb Z_{n/p_1}$ where $$ r=1+\sum_{i=1}^{m} \Big( \alpha_i(rb(\mathbb{Z}_{p_i},1)-2) \Big)-rb(\mathbb Z_{p_1},1)+2.$$
Therefore, by Lemma \ref{IStep}, there exists a rainbow-free $r+rb(\mathbb Z_{p_1},1)-2$ coloring of $\mathbb Z_n$.
Thus, by induction $$rb(\mathbb{Z}_n,1) \geq 2+\sum_{i=1}^{m} \Big( \alpha_i(rb(\mathbb{Z}_{p_i},1)-2) \Big).$$
\end{proof}

\subsection{Upper Bound}\label{SectionSchurUpperBound}
To establish the upper bound for $rb(\mathbb{Z}_{n}, 1)$, we consider residue classes  and their corresponding residue palettes under $c$.

\begin{lemma} \label{LimitedColors}
Let $R_0, R_1, \dots, R_{t-1}$ be the residue classes modulo $t$ for $\mathbb{Z}_{st}$, and let $P_0, P_1, \cdots, P_{t-1}$ be the corresponding residue palettes under rainbow-free $c$.
Then $|P_i \setminus P_0| \leq 1$ for $1\leq i \leq t-1$.
\end{lemma}

\begin{proof}
Assume that $|P_i \setminus P_0| \geq 2$.
Then $R_i$ must contain at least two elements which receive colors that do not appear in $P_0$.
Without loss of generality, let $G$ and $B$ denote two colors in $P_i\setminus P_0$.
Then there exists two integers $m$ and $n$ such that $c(mt + i)= G$ and $c(nt + i)=B$.
Consider the Schur triple $(mt-nt, nt + i , mt + i)$.
Notice that $mt-nt \equiv 0 \mod t$, $c(mt-nt)\neq G , B$.
Thus, we have a rainbow triple under $c$ in $\mathbb{Z}_{st}$, which is a contradiction.
Therefore, $|P_i \setminus P_0| \leq 1$ for $1\leq i \leq t-1$.
\end{proof}

Lemma \ref{LimitedColors} lets us create a well-defined reduction of a coloring of $\mathbb Z_st$ to a coloring of $\mathbb Z_t$.

\begin{lemma}\label{stProjection}
Let $s$ and $t$ be positive integers. 
Let $R_0, R_1, \dots, R_{t-1}$ be the residue classes modulo $t$ for $\mathbb{Z}_{st}$ with corresponding residue palettes $P_i$.
Suppose $c$ is a coloring of $\mathbb{Z}_{st}$ where $|P _i\setminus P_0| \leq 1$.
Let $\hat{c}$ be a coloring of $\mathbb{Z}_t$ given by 
\[ \hat{c}(i):=
\begin{cases} 
   P_i\setminus P_0 & \textrm{if } |P_i\setminus P_0| = 1\\
   \alpha & \textrm{otherwise}
\end{cases}
\]
where $\alpha \not \in P_i$ for $0 \leq i \leq t$.
If $\hat{c}$ contains a rainbow Schur triple, then $c$ contains a rainbow Schur triple.
\end{lemma}

\begin{proof}
Suppose $(x_1, x_2, x_3)$ is a rainbow Schur triple in $\hat{c}$. 
Then, at least two of $x_1, x_2, x_3$ must receive a color  other than $\alpha$. We consider the following two cases.

\textbf{Case 1:} Neither $x_1$ nor $x_2$ receive color $\alpha$.

Without loss of generality, assume that $c(x_1)=G$ and $C(x_2)=B$. 
This implies that there exist $n,m$ such that $c(nt+x_1)=G$ and $c(mt+x_2)=B$. 
There is a Schur triple of the form $(nt+x_1, mt+x_2, (n+m)t+(x_1+x_2))$ in $\mathbb{Z}_{st}$.
Since $x_1+x_2 \equiv x_3 \mod t$, $(n+m)t+(x_1+x_2)$ is in the residue class $R_{x_3}$. 
As  $\hat{c}(x_3)\neq G,B$, we have $G,B\notin P_{x_3}$. 
Therefore, the triple $(nt+x_1, mt+x_2, (n+m)t+(x_1+x_2))$ is rainbow.

\textbf{Case 2:} One of $x_1$ or $x_2$ is colored $\alpha$.

Without loss of generality, assume that $c(x_1)=\alpha$, $c(x_2)=B$, and $c(x_3)=G$.
Then $c(nt+x_2)=B$ for some $n$, and $c(mt+x_3)=G$  for some $m$. 
There is a Schur triple of the form $((m-n)t+(x_3-x_2), nt+x_2, mt+x_3)$ in $\mathbb{Z}_{st}$.
Since $x_1+x_2 \equiv x_3 \mod t$, $(m-n)t+(x_3-x_2)$ is in the residue class $R_{x_1}$. 
As $\hat{c}(x_1)=\alpha$, we have $G,B\notin P_{x_1}$. 
Therefore, the triple $((m-n)t+(x_3-x_2), nt+x_2, mt+x_3)$ is rainbow.

Hence, if $\hat{c}$ has a rainbow Schur triple, then $c$ has a rainbow Schur triple.
\end{proof}

We use the coloring described in Lemma \ref{stProjection} to prove an upper bound for $rb(\mathbb{Z}_{st}, 1)$.

\begin{prop}\label{UpperBound}
Let $s$ and $t$ be positive integers.
Then $rb(\mathbb{Z}_{st},1) \leq rb(\mathbb{Z}_s,1)+rb(\mathbb{Z}_t,1)-2$.
\end{prop}
\begin{proof}
Let $c$ be an exact $r$-coloring of $\mathbb{Z}_{st}$, and let $\hat{c}$ be a coloring constructed from $c$ as in Lemma \ref{stProjection}.
Notice that the set of colors used in $c$ is comprised of the colors in $R_0$ and each color used in $\hat{c}$ other than $\alpha$.
Thus,  $r=|P_0|+|\hat{c}|-1$, where $|\hat{c}|$ is the number of colors appearing in $\hat{c}$. 
If $c$ is a rainbow-free coloring of $\mathbb{Z}_{st}$, then $R_0$ is a rainbow-free coloring of $\mathbb{Z}_s$. Thus, $|P_0| \leq rb(\mathbb{Z}_s, 1)-1$.
Also, $\hat{c}$ is a rainbow-free coloring of $\mathbb{Z}_t$, so $|\hat{c}| \leq rb(\mathbb{Z}_t, 1)-1$.
Thus, $r \leq rb(\mathbb{Z}_s, 1)+rb(\mathbb{Z}_t, 1)-3$.
If we let $c$ be the maximum rainbow-free coloring of $\mathbb{Z}_{st}$, then $r=rb(\mathbb{Z}_{st},1)-1$.
This shows that $rb(\mathbb{Z}_{st}, 1) \leq rb(\mathbb{Z}_s, 1)+rb(\mathbb{Z}_t, 1)-2$.
\end{proof}

Using both the upper bound we  just established and the lower bound established in Proposition \ref{SchurLowerbd} of Section \ref{SectionSchurLowerBound}, we prove Theorem \ref{SchurFactorization}.

\begin{proof} [Proof of Theorem \ref{SchurFactorization}]
Recursively applying Proposition \ref{UpperBound} to prime factors of $n$ yields
$$rb(\mathbb{Z}_n,1) \leq 2+\sum_{i=1}^{m} \Big(\alpha_i(rb(\mathbb{Z}_{p_i},1)-2) \Big).$$
Since this is identical to the lower bound from Proposition \ref{SchurLowerbd} in Section \ref{SectionSchurLowerBound}, we can conclude
$$rb(\mathbb{Z}_n, 1) = 2+\sum_{i=1}^{m} \Big(\alpha_i(rb(\mathbb{Z}_{p_i},1)-2) \Big).$$
\end{proof}

\section{Triples for $x_1+x_2=px_3$, $p$ prime}
\label{pPrime}
Section \ref{pPrime} is dedicated to proving Theorem \ref{GenKFormula}. In Section \ref{pneqq}, we establish exact values for $rb(\mathbb Z_q,p)$ where $p\neq q$ are prime. Finding an exact value for $rb(\mathbb{Z}_p,p)$ is more difficult, and is the subject of Section \ref{primepowers}. Some properties of rainbow-free colorings of $\mathbb Z_q$ are used in the construction of the general lower bound in Section \ref{SectionGenKLowerBound}.
The equivalent upper bound is proved in \ref{SectionGenKUpperBound}. Combining Sections \ref{SectionGenKLowerBound} and \ref{SectionGenKLowerBound} proves Theorem \ref{GenKFormula}. 

\subsection{Exact values for $rb(\mathbb Z_q,p)$, $p\neq q$ prime}
\label{pneqq}
Lemmas \ref{1}, \ref{2}, \ref{3}, \ref{4} establish the upper bound $rb(\mathbb Z_q,p)\leq 4$. These lemmas are proven by assuming that there exists a rainbow-free $r$-coloring $c$ with $r\geq 4$, and reducing $c$ to a $3$-coloring $\hat c$. In each case, we find that $\hat c$ does not conform to the structure of a rainbow-free $3$-coloring outlined in Theorem \ref{3ColorPrimes} proven in \cite{LM}. For convenience, we include Theorem \ref{3ColorPrimes} and the necessary definitions from \cite{LM}.

For a subset $X\subseteq \mathbb Z_q^*$ and $a\in \mathbb Z_q^*$ define $aX:=\{ax\mid x\in X\}$, $X+a:=\{x+a\mid x\in X\}$, and $X-a:= X+(-a)$. We say the set $aX$ is the \emph{dilation} of $X$ by $a$.
Let $\langle x\rangle\leq \mathbb Z_q^*$ denote the subgroup multiplicatively generated by $x$. 
A subset $X\in \mathbb Z_q^*$ is \emph{$H$-periodic} if $X$ is the union of cosets of $H$, where $H\leq \mathbb Z_p^*$. 
In the case that $X$ is $\langle -1\rangle$-periodic, we say that $X$ is \emph{symmetric}. This coincides with the notion that $X$ is symmetric if and only if $X=-X$.

\begin{theorem}\label{3ColorPrimes}[\cite{LM}, Theorem 2]
A 3-coloring $\mathbb{Z}_q = A \cup B \cup C$ with $1 \leq |A| \leq |B| \leq |C|$ is rainbow-free for $x_1+x_2=kx_3$ if and only if, up to dilation, one of the following holds.

\begin{enumerate}
\item $A = \{0\}$ and both $B$ and $C$ are symmetric and $\langle k \rangle$-periodic subsets.
\item $A = \{1\}$ for 
\begin{itemize}
\item [(i)] $k=2 \mod q$, $(B-1)$ and $(C-1)$ are symmetric and $\langle 2 \rangle$-periodic subsets.
\item [(ii)] $k = -1 \mod q$, $(B\setminus \lbrace 2 \rbrace ) + 2^{-1}$, $(C\setminus \lbrace 2 \rbrace) + 2^{-1}$ are symmetric subsets.
\end{itemize}
\item $|A| \geq 2 $, for $k = -1 \mod q$ and $A,B,$ and $C$ are arithmetic progressions with difference $1$ such that $A = \lbrack a_1, a_2 - 1 \rbrack$, $B = \lbrack a_2, a_3 - 1 \rbrack$, and $C = \lbrack a_3, a_1 - 1 \rbrack$, with $(a_1 + a_2 + a_3) = 1$  or $2$.
\end{enumerate}
\end{theorem}

Suppose that $q\geq 5$ is prime. Let $c$ be a coloring of $\mathbb Z_q$ with color classes $C_1,\dots, C_r$ with $1\leq|C_1|\leq |C_2|\leq \cdots \leq |C_r|$ and $r\geq 4$.

\begin{obs}
If $C_1=\{0\}$ and $C_2=\{x\}$, then $(x,-x,0)$ is a rainbow triple for $x\neq 0$.
\end{obs} 

Therefore, if $c$ has two or more singleton color classes, we can assume that $\{0\}$ is not a color class. Furthermore, since dilation preserves the rainbow-free property, we can assume that if $|C_2|=1$, then $C_1=\{1\}$.

\begin{lemma}\label{1}
If $p \not\equiv -1 \mod q$ and $|C_2|=1$, then  $c$ admits a rainbow triple.
\end{lemma}

\begin{proof}
Consider the coloring $\hat c$ given by the color classes $ C_1, C_2, \bigcup_{i=3}^rC_i$. If $\hat{c}$ admits a rainbow triple, then $c$ also admits a rainbow triple and we are done. If $\hat{c}$ does not admit a rainbow triple, then $\hat{c}$ must conform to case 2.(i) in Theorem \ref{3ColorPrimes}. Therefore, $p\equiv 2\mod q$. In this case, triples satisfying $x_1+x_2=kx_3$ in $\mathbb Z_q$ are $3$-term arithmetic progressions. In \cite{BEHHKKLMSWY}, Proposition 3.5 establishes that $rb(\mathbb{Z}_q,2)\leq 4$. Therefore, there exists a rainbow triple under $c$.
\end{proof}

\begin{lemma}\label{2}
If $p\equiv -1\mod q$ and $|C_3|=1$, then $c$ admits a rainbow triple. 
\end{lemma}

\begin{proof}
Let $C_2=\{x\}, C_3=\{y\}$. For the sake of contradiction, assume that $c$ is rainbow free. 

If $x=2$, then $(x,-3,1)$ is a rainbow triple. The same argument for $y$ shows that $x,y\neq 2$. 

Consider the coloring $\hat c$ given by the color classes $ C_1, C_2, \bigcup_{i=3}^rC_i$. Then by Theorem \ref{3ColorPrimes} we must have $C_2\setminus\{2\}+2^{-1}$ is symmetric and so $x+2^{-1}=-2^{-1}-x$. Solving for $x$ gives that $x=-2^{-1}$.  Considering the coloring given by $C_1, C_3, C_2\cup \bigcup_{i=4}^rC_i$ gives that $y=-2^{-1}$, which is a contradiction.
\end{proof}

\begin{lemma}\label{3}
If $p\not\equiv -1\mod q$, and $|C_2|\geq 2$, then $c$ admits a rainbow triple.
\end{lemma}

\begin{proof}
 For the sake of contradiction, suppose that $c$ does not admit a rainbow triple. Consider the coloring $\hat c$ given by $C_1\cup C_2, C_3, \bigcup_{i=4}^r C_i$. Since $|C_3| \geq |C_2| \geq 2$, notice that $\hat c$ does not have a singleton color class and is rainbow-free. This contradicts Theorem \ref{3ColorPrimes}.
\end{proof}

\begin{lemma}\label{4}
If $p\equiv -1 \mod q$ and $|C_3|\geq 2$, then $c$ admits a rainbow triple.
\end{lemma}

\begin{proof}
For the sake of contradiction, suppose that $c$ does not admit a rainbow triple. There are two cases: $|C_2|\geq 2$, or $|C_2|=1$. 

\textbf{Case 1:} Assume that $|C_2|\geq 2$ and $C_1=\{x\}$. By Theorem \ref{3ColorPrimes}, the coloring $C_1\cup C_2, C_3,\bigcup_{i=4}^rC_i$ is of the form $$C_1\cup C_2=[a_1,a_2-1],$$$$C_3=[a_2,a_3-1],$$$$\bigcup_{i=4}^rC_i=[a_3,a_1-1].$$ 

$x$ is not adjacent to at least one of $C_3$ or $\bigcup_{i=4}^rC_i$. Without loss of generality, assume $x$ is not adjacent to $C_3$ (the other case follows the same argument). Consider the coloring $\hat c$ given by $C_2, C_1\cup C_3, \bigcup_{i=4}^rC_i$. Notice that $\hat c$ can only be dilated by $1$ or $-1$ to preserve the interval structure of $\bigcup_{i=4}^rC_i$. However, dilating by $1$ or $-1$ will not make $C_1\cup C_3$ an arithmetic progression with difference $1$. This is a contradiction. 

\textbf{Case 2:} Assume that $|C_2|=1$. Consider the coloring $\hat c$ given by $C_1\cup C_2, C_3,\bigcup_{i=4}^rC_i$. By Theorem \ref{3ColorPrimes}, $\hat c$ is of the form $$C_1\cup C_2=[a_1,a_2-1],$$$$C_3=[a_2,a_3-1],$$$$\bigcup_{i=4}^rC_i=[a_3,a_1-1]$$ with $a_1+a_2+a_3\in \{1,2\}$. Since every set is an arithmetic progression with difference $1$, $a_2-1=a_1+1$. This implies that $a_3\in\{-2a_1-1, -2a_1\}$.  This implies that $c(-2a_1-1)\neq c(a_1), c(a_1+1)$. Therefore, triple $(-2a_1-1, a_1, a_1+1)$ is rainbow, which is a contradiction.
\end{proof}

\begin{proof}[Proof of Theorem \ref{thmpneqq}]
By Lemmas \ref{1}, \ref{2}, \ref{3}, and \ref{4}, we know that $rb(\mathbb{Z}_q,p)\leq 4$. Therefore, it suffices to show that there exists a rainbow-free $3$-coloring of $\mathbb Z_q$ if and only if $p,q$ do not satisfy either condition 1 or 2. First we will prove that if there exists a rainbow-free $3$-coloring, then $p,q$ do not satisfy conditions 1 and 2.

Let $c$ be a rainbow-free $3$-coloring. There are two cases,  $p\not\equiv -1\mod q$ or $p\equiv -1\mod q$. 

\textbf{Case 1:} By Theorem \ref{3ColorPrimes}, either $0$ is uniquely colored, or $p\equiv 2 \mod q$.

Suppose $0$ is uniquely colored and $c(1)=R$. Notice that if $c(x)=R$, then $c(px)=R$ and $c(-x)=R$. If $p,q$ satisfy either 1 or 2, then $\{p^i,-p^i\mid i\in \mathbb Z\}= \mathbb Z_q^*$, which contradicts the fact that $c$ is a $3$-coloring. 

Suppose $p\equiv 2\mod q$. Then neither 1 nor 2 are satisfied by Theorem 3.5 in \cite{JLMNR}.

\textbf{Case 2:} Suppose $p\equiv -1\mod q$. Then $|p|=2$. If $(q-1)/2$ is odd, then $(q-1)/2\neq 2$. Therefore, neither 1 nor 2 are satisfied. 

To prove the reverse direction, suppose that $p,q$ do not satisfy either 1 or 2. Let $c$ be given by $$C_1=\{0\},C_2=\{p^i, -p^i\mid i\in \mathbb Z\}, C_3=\mathbb Z_q^*\setminus C_2.$$ Since $p,q$ do not satisfy either 1 or 2, $C_3$ is non-empty. Notice that any rainbow triple must contain $0$ and some element $y\in C_2$. However, if $0,y, z$ is a triple, then $z\in C_2$. Therefore, $c$ is rainbow-free.

\end{proof}

The following corollary is used in Section \ref{SectionGenKLowerBound} to prove a general lower bound for $rb(\mathbb Z_n,p)$.

\begin{cor}\label{uniquecolorcor}
There exists a maximum rainbow-free coloring of $\mathbb{Z}_q$ where $0$ is uniquely colored and the color classes are symmetric.
\end{cor}

\subsection{Exact values for $rb(\mathbb{Z}_{p^\alpha},p)$, $p$ prime}
\label{primepowers}
In order to determine the rainbow numbers for equations of the form $x_1+x_2=px_3$ for prime $p\geq 3$ we still need to determine $rb(\mathbb Z_{p^\alpha},p)$ for $\alpha\geq 1$. We will prove Theorem \ref{prime} using induction. Observation \ref{p=3} and Propositions \ref{K=P}, \ref{primelb3}, and \ref{primelb5}  provide the lower bound and base case for our induction argument. Lemmas \ref{samecolors} and \ref{zeromono} provide the basic structure of a rainbow-free coloring of $\mathbb{Z}_{p^\alpha}$. Lastly, Lemmas \ref{case12}, and \ref{atleast2} exploit the structure to derive a contradiction by forcing a rainbow triple. Throughout this section, for $0\leq k\leq p-1$, recall that the $k^{th}$ residue class mod $p$ is the set $R_k = \{j \in \mathbb{Z}_{p^\alpha}: j\equiv k \mod p\}$ and that the $k^{th}$ residue palette $P_k$ is the set of colors which appear on $R_k$.

\begin{obs}\label{p=3}
Notice $rb(\mathbb Z_3,3)=3$ and $rb(\mathbb Z_9,3)=4$.
\end{obs}

\begin{prop}\label{K=P}
Let $p\geq 3$ be prime. Then $rb(\mathbb Z_p,p)=\frac{p+1}{2}+1.$
\end{prop}

\begin{proof}
To prove the lower bound, consider the following coloring:

$$c(x)=\begin{cases} x & 0\leq x\leq \frac{p+1}{2}\\
-x& \text{otherwise}\end{cases}.$$
Notice that $c(x)=c(-x)$ for all $x\in \mathbb Z_p$. Furthermore, if $(x_1,x_2,x_3)$ is a triple, then $x_1=-x_2$. Thus, $c$ is a rainbow-free $\frac{p+1}{2}$ coloring, and $rb(\mathbb{Z}_p,p)>\frac{p+1}{2}$. 

To prove the upper bound, assume that $c$ is an $\frac{p+1}{2}+1$ coloring of $\mathbb Z_p$. By the pigeonhole principle, there exists $x\in \mathbb{ Z}_p$ such that $x\neq 0$ and $c(x)\neq c(-x)$. Since $p\geq 3$, $x\neq -x$, and there exist $y\neq x,-x$ such that $c(y)\neq c(x),c(-x)$. Therefore, $(x,-x,y)$ is a rainbow-triple, and $rb(\mathbb Z_p,p)\leq \frac{p+1}{2}+1$.
\end{proof}

For the rest of the section, we will assume that $\alpha\geq 2$. 

\begin{prop}\label{primelb3}
For $\alpha\geq 2$, 
$$rb(\mathbb{Z}_{3^\alpha},3) > 3.$$
\end{prop}

\begin{proof}
Suppose that $\alpha\geq 3$ and $\bar c$ is a rainbow-free $3$-coloring of $\mathbb Z_9$. Let $c$ be a $3$-coloring of $\mathbb Z_{p^\alpha}$ given by $c(i):=\bar c(i\mod 9)$. Assume that $x_1,x_2,x_3$ is a triple in $\mathbb Z_{3^\alpha}$. Then $x_1,x_2,x_3$ is a triple in $\mathbb Z_9$ and cannot be rainbow. 
\end{proof}

\begin{prop}\label{primelb5}
For prime $p\geq 5$ and $\alpha\geq 1$, 
$$ rb(\mathbb{Z}_{p^\alpha},p)\geq
\frac{p+1}{2}+1.$$
\end{prop}

\begin{proof}
Color all of $R_i, R_{p-i}$ color $i$ for $0\leq i \leq \frac{p+1}{2}$. Suppose $x_1+x_2 = px_3$ and $x_1\equiv j\mod p$ for $0\leq j\leq p-1$. Then $x_2\equiv p-j\mod p$, and $x_1,x_2,x_2$ is not rainbow. 
\end{proof}

\begin{lemma}\label{samecolors}
If $c$ does not admit a rainbow triple, then 
$$P_i=P_{p-i}$$ when $0<i<p$. 
\end{lemma}

\begin{proof}
For the sake of contradiction, suppose that there exists $0<i<p$ with $G\in P_i\setminus P_{p-i}$. Then there exists an element $px+i$ with color $G$ in $R_i$. Let $py+p-i$ be an element in $R_{p-i}$. Notice that 
\begin{align*}
x_1&=p(py-x+p-1-i)+p-i\\
x_2&=px+i\\
x_3&=py+p-i
\end{align*}
is a triple. Since $G\notin P_{p-i}$, we have $c(x_3)=c(x_1)$. Furthermore, $x_1-x_3=p(py-x+p-1-i)+p-i-py-p+i=p(y(p-1)-x+p-1)$. Since $py+p-i$ was arbitrary, we can choose $y$ so that $y(p-1)-x+p-1\not\equiv 0\mod p$. Since $y(p-1)-x+p-1\not\equiv 0\mod p$, we know that $y(p-1)-x+p-1$ is an additive generator of $\mathbb{Z}_{p^{\alpha-1}}$. This implies that $P_{p-i}=\{B\}$. 

Let $pz+j$ be an element with $c(pz+j)\notin \{G,B\}$. Then 
\begin{align*}
x_1&=p(pz-x+j-1)+p-i\\
x_2&=px+i\\
x_3&=pz+j\\
\end{align*}
is a rainbow triple, which is a contradiction. 
\end{proof}

Notice that by Lemma \ref{samecolors}, it is sufficient to only consider the structure of $R_i$ for $0<i<\frac{p+1}{2}$.

\begin{lemma}\label{zeromono}
Suppose $c$ does not admit a rainbow triple. If there exists $0<i<p$ such that $|P_i\setminus P_0|\geq 1$, then $|P_0|=1$. 
\end{lemma}

\begin{proof}
Since $c$ does not admit a rainbow triple, $P_i=P_{p-i}$. Without loss of generality, suppose that $G\in P_i\setminus P_0$ and let $c(pa_1+i)=c(pa_2+p-i)=G$. Let $pb\in R_0$ be arbitrary. Consider the following triple:
\begin{align*}
x_1&=pb\\
x_2&=p(pa_1+i-b)\\
x_3&=pa_1+i.
\end{align*} 
Since $c$ is rainbow-free, $c(x_1)=c(x_2)$. Next, consider the following triple:
\begin{align*}
x'_1&=p(pa_1+i-b)\\
x'_2&=p(pa_2+p-i-pa_1-i+b)\\
x'_3&=pa_2+p-i.
\end{align*}
Since $c$ is rainbow-free, $c(x'_1)=c(x'_2)$. This implies that $$c(pb)=c(p(pa_2+p-i-pa_1-i+b)).$$ Notice that difference in position between $x'_2$ and $pb$, given by $pa_2+p-i-pa_1-i+b-b$, does not depend on $b$. Furthermore, $pa_2+p-i-pa_1-i+b-b$ is relatively prime to $p^{\alpha-1}$. Therefore, all elements in $R_0$ receive the same color. 
\end{proof}

\begin{lemma}\label{case12}
Let $p$ be prime with $p\geq 5$. If there exists $0<i<\frac{p+1}{2}$ such that $|P_i\setminus P_0|\geq 2$ and $G\notin P_i\cup P_0$, then $c$ admits a rainbow triple.
\end{lemma}

\begin{proof}
For the sake of contradiction, suppose that $c$ does not admit a rainbow triple. Since $p\geq 5$ and $|P_0|=1$, there exists $j\neq i$ such that $0<j<p$ and $G\in P_j\setminus (P_i\cup P_0)$. By Lemma \ref{samecolors}, $P_j=P_{p-j}$ and $P_i=P_{p-i}$. Let $c(pa_1+j)=c(pa_2+p-j)=G$. Let $pb+i\in R_i$ be arbitrary. Consider the following triple:
\begin{align*}
x_1&=pb+i\\
x_2&=p(pa_1+j-b-1)+p-i\\
x_3&=pa_1+j.
\end{align*}
Then $c(x_1)=c(x_2)$. Next consider the following triple:
\begin{align*}
x'_1&= p(pa_1+j-b-1)+p-i\\
x'_2&=p(pa_2+p-j-pa_1-j+b)+i\\
x'_3&=pa_2+p-j
\end{align*}
Then $c(x_1')=c(x'_2)$. This implies that $$c(pb+i)=c(p(pa_2+p-j-pa_1-j+b)+i).$$
Notice that the difference in position between $x'_2$ and $pb+i$, given by $pa_1+p-j-pa_1-j+b-b$, does not depend on $b$. Furthermore, $pa_2+p-j-pa_1-j+b-b$ is relatively prime to $p^{\alpha-1}$. Therefore, all elements in $R_i$ receive the same color. This is a contradiction, since $|P_i|\geq 2$. 
\end{proof}

\begin{lemma}\label{atleast2}
If $p\geq 5$, $\mathbb{Z}_{p^\alpha}$ is colored with at least $4$ colors, and there exists $0<i<\frac{p+1}{2}$ with $\text{Im}(c)=P_i\cup P_0$ and $|P_i\setminus P_0|\geq 2$, then $c$ admits a rainbow triple.
\end{lemma}

\begin{proof}
For the sake of contradiction, suppose that $c$ does not admit a rainbow triple. By Lemma \ref{zeromono}, let $P_0=\{R\}$. By Lemma \ref{samecolors}, $P_i=P_{p-i}$. Since $P_i$ contains all colors except possibly $R$, there exists $a,b,d$ such that $c(pa+i)=G$, $c(pb+p-i)=B$ and $c(pd+i)=B$. Consider the following triple:
\begin{align*}
x_1&=pa+i\\
x_2&=p(pb+p-i-a-1)+p-i\\
x_3&=pb+p-i.
\end{align*}
Then $c(x_2)\in \{B,G\}$. Let $x\in \{a,d\}$ such that $c(px+i)\neq c(x_2)$ and consider the following triple:
\begin{align*}
x'_1&=p(pb-p-i-a-1)+p-i\\
x'_2&=p(px-pb+p+2i+a)+i\\
x'_3&=px+i.
\end{align*}
Notice that $c(x'_2)\in\{B,G\}$. Furthermore, the difference in position between $x'_2$ and $pa+i$, given by $px-pb+p+2i\equiv 2i\mod p$, does not depend on $a,b,d$ modulo $p$. Therefore, for any $x\in \mathbb Z_p$ there exists $a\equiv x$ such that $c(pa+i)\in \{B,G\}$.

Since $P_{p-i}$ contains all colors of $c$ except for possibly $R$, there exists $y$ such that $c(py+p-i)=Y$. Select $a\equiv -1-y\mod p$ such that $c(pa+i)\in \{B,G\}$. Then the triple $(py+p-i, pa+i, a+y+1)$ is rainbow since $a+y+1\in R_0$.
\end{proof}

\begin{proof}[Proof of Theorem \ref{prime}]
Proposition \ref{primelb3} provides the lower bound for $p=3$, $\alpha\geq 2$. Observation \ref{p=3} covers the case when $p=3,\alpha=1,2$.

We will proceed by induction on $\alpha.$ Suppose that $rb(\mathbb Z_{p^{\alpha-1}},3)=4$ for some $\alpha \geq 3$. Let $c$ be a $4$ coloring of $\mathbb{Z}_{3^\alpha}$. For the sake of contradiction, suppose that $c$ does not admit a rainbow triple. If $|P_0|=4$, then $c$ admits a rainbow triple by the induction hypothesis. Therefore, $|P_0|\leq 3$ and there exits $0<i<p$ such that $|P_i\setminus P_0|\geq 1$. By Lemma \ref{zeromono}, $|P_0|=1$. This implies that $\text{im}(c)=|P_i\setminus P_0|$. By Lemma \ref{atleast2}, $c$ admits a rainbow triple. This completes the case when $p=3$. 
 
Let $p\geq 5$. With Proposition \ref{K=P} as the base case, we will proceed by induction on $\alpha$. Suppose that $rb(\mathbb{Z}_{p^{\alpha-1}},p)=\frac{p+1}{2}+1$ for some $\alpha \geq 2$. For the sake of contradiction, suppose that $c$ does not admit a rainbow triple. If $|P_0|=\frac{p+1}{2}+1$, then $c$ admits a rainbow triple by the induction hypothesis. Therefore, $|P_0|\leq \frac{p+1}{2}$ and there exists $0<j<p$ such that $|P_j\setminus P_0|\geq 1$. By Lemma \ref{zeromono}, $P_0=\{R\}$. By the pigeon hole principle, there exists $0<i<\frac{p+1}{2}$ such that $|P_i\setminus P_0|\geq 2$. Notice that one of the following must hold:
\begin{enumerate}
\item $G\notin P_i\cup P_0$ for some color $G\neq R$,
\item $\text{im}(c)=P_i\cup P_0$.
\end{enumerate}
Therefore, by Lemmas \ref{case12} and \ref{atleast2}, $c$ must admit a rainbow triple. This completes the case when $p\geq5$.
\end{proof}

\subsection{Lower bound for $rb(\mathbb{Z}_n,p)$, $p$ prime} \label{SectionGenKLowerBound}
Since $p$ is the coefficient of the equation that we are considering, we will use $q$ to denote a prime other than $p$.
Using values for $rb(\mathbb{Z}_q, k)$, we establish a lower bound for $rb(\mathbb{Z}_n, p)$.
In order to proceed in a similar manner as with the Schur equation, two lemmas about the structure of triples are necessary. 

\begin{lemma} \label{SweetDivisonTrick}
If $x_1 + x_2 = k x_3$ is a triple in $\mathbb{Z}_n$ where $m|x_1,x_2,x_3$ for some $m|n$, $m,n\in \mathbb{Z}$, then there exists a triple of the form $x_1/m + x_2/m = k x_3/m$ in $\mathbb{Z}_\frac{n}{m}$.
\end{lemma}
\begin{proof}
By definition $x_1 + x_2 = k x_3$ in $\mathbb{Z}_n$ implies:
\begin{align*}
x_1 + x_2 = qn + r \\
k x_3 = tn + r
\end{align*}

Divide both equations by $m$ to get:
\begin{align*}
\frac{x_1}{m} + \frac{x_2}{m} = q\frac{n}{m} + \frac{r}{m} \\
k \frac{x_3}{m} = t\frac{n}{m} + \frac{r}{m}
\end{align*}

Now we must check that $\frac{r}{m}$ is an integer. Since $m|(x_1 + x_2 - qn)$, we know $m|r$.
By definition, this means there exists a triple of the form $x_1/m + x_2/m = x_3/m$ in $\mathbb{Z}_\frac{n}{m}$.
\end{proof}

Next, we show that $q$ cannot divide exactly two terms of a triple.

\begin{lemma}\label{qtriple}
Let $(x_1, x_2, x_3)$ be a triple of the form $x_1 + x_2 = k x_3$ in $\mathbb{Z}_{qn}$. 
If $q$ is relatively prime to $k$ and $q$ divides two of the terms in $(x_1, x_2, x_3)$ then $q$ must divide the third term in $(x_1, x_2, x_3)$.
\end{lemma}
\begin{proof}
We consider the case where $q$ divides $x_1$, $x_2$ and the case where $q$ divides $x_1$, $x_3$.

\textbf{Case 1: }
Assume $q$ divides $x_1$, $x_2$.
By definition the equation $x_1 + x_2 = k x_3$ in $\mathbb{Z}_{qn}$ means:
\begin{align*}
x_1 + x_2 = c_1 qn + r \\
k \cdot x_3 = c_2 qn + r
\end{align*}
We rearrange the first equation to get $q$ divides $x_1 + x_2 - c_1 qn$ which implies that $q$ divides $r$.
Thus $q$ divides $c_2 qn + r$ which mplies $q$ divides $k x_3$.
We know $q$ and $k$ are relativity prime, therefore $q$ must divide $x_3$.

\textbf{Case 2: } Similarly, assume $q$ divides $x_1$, $x_3$.
By definition the equation $x_1 + x_2 = k x_3$ in $\mathbb{Z}_{qn}$ means:
\begin{align*}
x_1 + x_2 = c_1 qn + r \\
k \cdot x_3 = c_2 qn + r
\end{align*}
From the second equation we get $q$ divides $k x_3 - c_2 qn$ which implies that $q$ divides $r$.
Thus $q$ divides $x_1 - c_1 \cdot qn - r $ which implies $q$ divides $x_2$.
\end{proof}

Notice that Lemmas \ref{SweetDivisonTrick} and \ref{qtriple} are stated for the equation $x_1+x_2=kx_3$ without the stipulation that $k$ is prime. We can use the above lemmas to find our lower bound. 

\begin{lemma}\label{IStep2}
Let $q,t$ be positive integers with $q$ prime,  and $q\neq p$. If there exists a rainbow-free $r$-coloring of $\mathbb{Z}_t$, then there exists a rainbow-free $(r+rb(\mathbb Z_q,p)-2)$-coloring of $\mathbb{Z}_{qt}$.
\end{lemma}

\begin{proof}
Let $q,t\in\mathbb{Z}$ such that $q$ is prime, and $q\neq p$. 
Let $\hat{c}$ be a rainbow-free $r$-coloring for $\mathbb{Z}_{t}$ and let $\bar{c}$ be a maximum coloring of $\mathbb{Z}_q$ such that 0 is uniquely colored and the other color classes are symmetric subsets, as described in Corollary \ref{uniquecolorcor}.
Let $c$ be an exact $(r+1)$-coloring of $\mathbb{Z}_{qt}$ if $rb(\mathbb{Z}_q, p)=3$ or an exact $(r+2)$-coloring of $\mathbb{Z}_{qt}$ if $rb(\mathbb{Z}_q, p)=4$ as follows:
$$c(x)=\begin{cases}
      \hat{c}(\frac{x}{q}) & x \equiv 0 \mod q\\ 
      r+\bar{c}(x\mod q) & \text{otherwise}.
   \end{cases}
$$

Since $q$ and $p$ are distinct primes, $q$ and $p$ are relatively prime.
By Lemma \ref{qtriple}, since $q$ is relatively prime to $p$, $q$ cannot divide exactly two of the terms in $(x_1,x_2,x_3)$ for the equation $x_1+x_2=px_3$.
Therefore, for all triples in $\mathbb{Z}_{qt}$, $q$ can divide all three elements, no elements, or exactly one element of the triple. 

\textbf{Case 1:} If $q$ divides all three terms in $(x_1,x_2,x_3)$, then by the constructions of $c$, the triple has the same colors as the triple $(\frac{x_1}{q},\frac{x_2}{q},\frac{x_3}{q})$ in $\hat{c}$. 
By Lemma \ref{SweetDivisonTrick}, if $(x_1,x_2,x_3)$ is a triple in $\mathbb{Z}_{qt}$ and $q|x_1,x_2,x_3$, then $(\frac{x_1}{q},\frac{x_2}{q},\frac{x_3}{q})$ is a triple in $\mathbb{Z}_t$. 
Thus, since $\hat{c}$ is a rainbow-free coloring, triples where all three elements are divisible by $q$ cannot be rainbow in $c$.

\textbf{Case 2:} Suppose $q$ divides none of the terms in $(x_1,x_2,x_3)$, there is a maximum of two colors added on terms not divisible by $q$. 
Thus, there are at most two colors coloring the elements in any such triple, and triples of the form $(x_1,x_2,x_3)$ with each $x_i$ not divisible by $q$ are not rainbow.

\textbf{Case 3:} Suppose $q$ divides exactly one of $(x_1,x_2,x_3)$. 
First assume $q$ divides $x_1$.
Notice that if $x_1 + x_2 \equiv p x_3 \mod qt$ then $x_1 + x_2 \equiv p x_3 \mod q$.
Since $0$ is uniquely colored in $\bar{c}$, the rainbow-free coloring of $\mathbb{Z}_q$, any triple in $\mathbb{Z}_q$ of the form $0 + x_2 \equiv px_3 \mod q$ is colored so that $x_2$ and $x_3$ receive the same color.
In this case, $c(x_2) =r+ \bar{c}(x_2 \mod q)$ and $c(x_3) =r+ \bar{c}(x_3 \mod q)$, so $(x_1,x_2,x_3)$ is not rainbow under $c$. If $q$ divides either $x_2$ or $x_3$ the argument proceeds the same way.
\end{proof}

\begin{prop} \label{GenKLowerBound}
Let $p$ be prime and let $n$ be an integer with prime factorization $n=p^\alpha \cdot q_1^{\alpha_1}\cdot q_2^{\alpha_2}\cdots q_m^{\alpha_m}$ where $q_i$ is prime, $q_i \neq q_j$ for $i \neq j$ and $\alpha_i \geq 0$. Then, 
$$rb(\mathbb{Z}_n, p) \geq rb(\mathbb{Z}_{p^\alpha}, p) + \sum_{i=1}^m \Big( \alpha_i (rb(\mathbb{Z}_{q_i}, p)-2) \Big)$$
\end{prop}
\begin{proof}
If $n$ is a power of $p$, then there is nothing to show. Suppose that the claim holds true for $n$ where $n$ has $N$ prime factors that are not $p$. 

Assume that $n=p^\alpha \cdot q_1^{\alpha_1}\cdot q_2^{\alpha_2}\cdots q_m^{\alpha_m}$ where $\alpha_1+\dots+\alpha_m=N+1$. By the induction hypothesis, there exists a rainbow-free $r$-coloring of $\mathbb{Z}_{n/q_1}$ where $$r=rb(\mathbb{Z}_{p^\alpha}, p) + \sum_{i=1}^m \Big( \alpha_i (rb(\mathbb{Z}_{q_i}, p)-2) \Big)-rb(\mathbb{Z}_{q_1},p)+2.$$ 

Therefore, by Lemma \ref{IStep2} there exists a rainbow-free $r+\mathbb Z_{q_1},p)-2$ coloring of $Z_n$. Thus, by induction
$$rb(\mathbb{Z}_{p^\alpha}, p) + \sum_{i=1}^m \Big( \alpha_i (rb(\mathbb{Z}_{q_i}, p)-2) \Big).$$
\end{proof}

\subsection{Upper bound for $rb(\mathbb{Z}_n,p)$, $p$ prime} \label{SectionGenKUpperBound}
In this section we prove an upper bound matching Proposition \ref{GenKLowerBound}. The proof of the upper bound uses the following lemmas.

\begin{lemma}\label{nottoobig}
Suppose $c$ is a rainbow-free coloring of $\mathbb{Z}_{qt}$ for $x_1 + x_2 = p x_3$ where $t$ is some positive integer and $q \neq p$ is prime.
Let $R_0, \cdots, R_{t-1}$ be the residue classes modulo $t$ of $\mathbb{Z}_{qt}$, with corresponding color palettes $P_0, \cdots, P_{t-1}$.
Let $j$ be an index such that $|P_j|\geq |P_i|$ for all $0\leq i\leq t-1$. Then $|P_i \setminus P_j| \leq 1$ for all $0 \leq i \leq t-1$.
\end{lemma}

\begin{proof}
For the sake of contradiction, assume that there exists $i$ such that $|P_i\setminus P_j|\geq 2$. This implies that there exists $tu+i$ and $tv+i$ with colors G and B respectively, that are not in $P_j$. Without loss of generality, $v>u$

First suppose that $P_{pi-j}\neq P_j$. There are two cases: either $P_{pi-j}$ has a color that is not in $P_j$, or $P_j$ has a color that is not in $P_{pi-j}$. 

\textbf{Case 1:} Suppose that $c(st+pi-j)\notin P_j$. Without loss of generality, $c(st+pi-j)\neq G$. Then 
\begin{align*}
x_1&=ts+pi-j\\
x_2&=ptu+-ts+j\\
x_3&=tu+i
\end{align*}
is a rainbow triple.

\textbf{Case 2:} Suppose that $c(ts+j)\notin P_{pi-j}$. Then 
\begin{align*}
x_1&=ts+j\\
x_2&=ptu-ts+pi-j\\
x_3&=tu+i
\end{align*}
is rainbow. 

Since $c$ is assumed to be rainbow-free, both cases result in a contradiction. Therefore, $P_j=P_{pi-j}$.

Let $ts+j\in R_j$. Since $c$ is rainbow-free, $c(ptu-ts+pi-j)=c(ts+j)$. Similarly, the triple 
\[
\{t(pu-s) + pi-j, t(pv-pu+s)+j, tv+i\}
\]
shows that $c(ptv-ptu+ts+j)=c(ptu-ts+pi-j)=c(ts+j)$. Notice that the difference of position between $ptv-ptu+ts+j$ and $ts+j$ in $R_j$ is $p(v-u)$. Since $p\neq q$ is prime and $v-u<q$, we know that $p(v-u)$ generates $\mathbb Z_q$. Therefore, $R_j$ is monochromatic; this contradicts the maximality of $|P_j|$.
\end{proof}

Lemma \ref{nottoobig} allows us to create a well-defined reduction of a coloring of $\mathbb Z_{qt}$ to a coloring of $\mathbb Z_t$.

\begin{lemma}\label{GenKProjection}
Let $t$ be a positive integer and $q \not = p$ be prime. 
Let $R_0, R_1, \cdots, R_{t-1}$ be the residue classes modulo $t$ for $\mathbb{Z}_{qt}$ with corresponding residue palettes $\{P_i\}$.
Let $j$ be an index such that $|P_j| \geq |P_i|$ for all $0 \leq i < t$.
Suppose $c$ is a coloring of $\mathbb{Z}_{qt}$ where $|P _i\setminus P_j| \leq 1$.
Let $\hat{c}$ be a coloring of $\mathbb{Z}_t$ such that: 
\[ \hat{c}(i):=
\begin{cases} 
   P_i\setminus P_j & \textrm{if } |P_i\setminus P_j| = 1\\
   \alpha & \textrm{otherwise}
\end{cases}
\]
If $\hat{c}$ contains a rainbow triple then $c$ contains a rainbow triple. 
\end{lemma}

\begin{proof}
Suppose that $(x_1,x_2,x_3)$ is a rainbow triple in $\mathbb{Z}_t$ under $\hat c$. There are two cases:$\hat c(x_3)=\alpha$, or $\hat c(x_3)\neq \alpha$.

\textbf{Case 1:} If $\hat c(x_3)=\alpha$, then $\alpha\neq \hat c(x_1),\hat c(x_2)$. Without loss of generality, suppose that $x_1$ and $x_2$ are colored $G$ and $B$, respectively. This implies that there exists $u,v$ such that $c(tu+x_1)=G$ and $c(tv+x_2)=B$. We must find integer $s$ such that $$u+v-ps\equiv \begin{cases} 1 \mod q& x_1+x_2\geq t\\
0\mod q & x_1+x_2<t\end{cases}.$$ Since $p$ and $q$ are relatively prime, we can alway solve for $s$. Therefore, there exists a rainbow triple in $\mathbb Z_{qt}$ under $c$.

\textbf{Case 2:} Assume $\hat c(x_3)\neq\alpha$. Without loss of generality, $\hat c(x_1)\neq \alpha$, and there exists $u,v$ such that $c(tu+x_1)=G$ and $c(tv+x_3)=B$ where $G,B\notin P_{x_2}$. Notice that $ptv-tu+px_3-x_1\in R_{x_2}$. Therefore, there exist a rainbow triple in $\mathbb Z_{qt}$ under $c$.
\end{proof}

\begin{prop} \label{GenKUpperBound}
Let $t$ be a positive integer, and let $q$ and $p$ be distinct primes.
Then $$rb(\mathbb{Z}_{qt},p) \leq rb(\mathbb{Z}_t,p) + rb(\mathbb{Z}_q,p) - 2.$$
\end{prop}
\begin{proof}
Let $c$ be a rainbow-free $r$-coloring of $\mathbb{Z}_{qt}$, and let $\hat{c}$ be a coloring constructed from $c$ as described in Lemma \ref{GenKProjection}.
Notice that the set of colors used in $c$ is comprised of the colors in $R_j$ and each color used in $\hat{c}$ other than $\alpha$. Thus, we know that $r=|P_j|+|\hat{c}|-1$, where $|\hat{c}|$ is the number of colors appearing in $\hat{c}$. 

Since $c$ is a rainbow-free coloring of $\mathbb{Z}_{qt}$, then $c|_{R_j}$ must be a rainbow-free coloring of $\mathbb{Z}_q$, so $|P_j| \leq rb(\mathbb{Z}_q,p)-1$.
Furthermore, $\hat{c}$ is a rainbow-free coloring of $\mathbb{Z}_t$, implying that $|\hat{c}| \leq rb(\mathbb{Z}_t,p)-1$.
Therefore, $r \leq rb(\mathbb{Z}_t,p)+rb(\mathbb{Z}_q,p)-3$.
If we let $c$ be the maximum rainbow-free coloring of $\mathbb{Z}_{qt}$, then $r=rb(\mathbb{Z}_{qt},p)-1$.
This shows that $rb(\mathbb{Z}_{qt},p) \leq rb(\mathbb{Z}_t,p) + rb(\mathbb{Z}_q,p)-2$.
\end{proof}

We can use Proposition \ref{GenKUpperBound} to find a matching upper bound for Proposition \ref{GenKLowerBound}.

\begin{proof} [Proof of Theorem \ref{GenKFormula}]
Recursively applying Proposition \ref{GenKUpperBound} for every prime factor $p_i \not = p$ of $n$  gives
$$rb(\mathbb{Z}_n, p) \leq rb(\mathbb{Z}_{p^{\alpha}}, p) + \sum_{i=1}^m \Big( \alpha_i(rb(\mathbb{Z}_{q_i}, p)-2) \Big).$$
Since this is identical to the lower bound from Proposition \ref{GenKLowerBound}, we can conclude
$$rb(\mathbb{Z}_n, p) = rb(\mathbb{Z}_{p^{\alpha}}, p) + \sum_{i=1}^m \Big( \alpha_i(rb(\mathbb{Z}_{q_i}, p)-2) \Big).$$
\end{proof}

\section*{Acknowledgements}

This research took place primarily at SUAMI at Carnegie Mellon University and the authors would like to thank the NSA for funding the program.




\begin{thebibliography}{00}
\bibitem{AF} M. Axenovich, and D. Fon-Der-Flaass, On rainbow arithmetic progressions. \emph{European Journal of Combinatorics} \textbf{11} (2004), no. 1, Research Paper 1, 7pp.

\bibitem{AM} M. Axenovich, and R.R. Martin, Sub-Ramsey numbers for arithmetic progressions. \emph{Graphs Comb.} \textbf{22} (2006), no. 1, 297-309.

\bibitem{B} F.A. Behrend, On sets of integers which contain no three terms in arithmetical progression. \emph{Proc. Nat. Acad. Sci. USA} \textbf{32} (1946), 331-332.

\bibitem{BSY} Z. Berikkyzy, A. Schulte, and M. Young. Anti-van der Waerden numbers of $3$-term arithmetic progressions. \emph{Electronic Journal of Combinatorics}, \textbf{24}(2): \#P2.39, (2017).

\bibitem{BEHHKKLMSWY} S. Butler, C. Erickson, L. Hogben, K. Hogenson, L. Kramer, R. Kramer, J. Lin, R. Martin, D. Stolee, N. Warnberg, and M. Young, Rainbow arithmetic progressions. \emph{Journal of Combinatorics}, \textbf{7} (4) (2016), 595-626. 

\bibitem{G} W.T. Gowers, A new proof of Szemer\'edi's theorem. \emph{Geom. Funct. Anal.} \textbf{11} (2001), n0. 3, 465-588.

\bibitem{JLMNR} V. Jungi\'c, J. Licht (Fox), M. Mahdian, J. Ne\v set\v ril, and R. Radoi\v ci\'c, Rainbow arithmetic progressions and anti-Ramsey results. \emph{Combin. Probab. Comput.} \textbf{12} (2003), no. 5-6, 599–620.

\bibitem{LM} B. Llano and A. Montenjano, Rainbow-free colorings for $x+y=cz$ in $\mathbb{Z}_p$. \emph{Discrete Mathematics}, \textbf{312} (17) (2012), 2566-2573.

\bibitem{Y} M. Young, Rainbow arithmetic progressions in finite abelian groups. To appear in \emph{Journal of Combinatorics}. arXiv:1603.08153 [math.CO].
\end{thebibliography}
\end{document}